\documentclass[11pt]{amsart}
\usepackage{amsmath,amsxtra,amsbsy,enumerate, latexsym, amsfonts, amssymb, amsthm, amscd, stmaryrd}
\usepackage{hyperref}
\usepackage{graphics,epsf,psfrag,pgfplots}
\usepackage{bbm, dsfont}
\usepackage{mathtools}
\usepackage{xcolor}
\usepackage{float}
\usepackage{tikz,pgfplots}
\usepackage[headings]{fullpage}
\usepackage{colonequals} 
\usepackage{microtype}
\usetikzlibrary{snakes}
\numberwithin{equation}{section}

\newcommand{\bbE}{{\ensuremath{\mathbb E}} }

\newcommand{\bbN}{{\ensuremath{\mathbb N}} }

\newcommand{\bbP}{{\ensuremath{\mathbb P}} }

\newcommand{\bbR}{{\ensuremath{\mathbb R}} }

\newcommand{\bbZ}{{\ensuremath{\mathbb Z}} }

\newcommand{\gb}{\beta}

\newcommand{\gk}{\kappa}

\newcommand{\gO}{\Omega}

\newcommand{\gL}{\Lambda}

\newcommand{\tf}{\textsc{f}}

\newcommand{\cA}{{\ensuremath{\mathcal A}} }

\newcommand{\cH}{{\ensuremath{\mathcal H}} }

\newcommand{\cN}{{\ensuremath{\mathcal N}} }

\newcommand{\cV}{{\ensuremath{\mathcal V}} }
\newcommand{\cZ}{{\ensuremath{\mathcal Z}} }

\newcommand{\bP}{{\ensuremath{\mathbf P}} }

\newcommand{\tP}{  \tilde {\mathbb{P}}}

\newcommand{\ind}{\mathbf{1}}

\newcommand{\lint}{\llbracket}
\newcommand{\rint}{\rrbracket}

\definecolor{darkred}{rgb}{0.7,0.1,0.1}

\definecolor{darkgreen}{rgb}{0.1,0.7,0.1}

\newtheorem{theorema}{Theorem}

\newtheorem{theorem}{Theorem}[section]
\newtheorem{lemma}[theorem]{Lemma}
\newtheorem{proposition}[theorem]{Proposition}

\newtheorem{rem}[theorem]{Remark}

\newcommand{\RN}[1]{%
  \textup{\uppercase\expandafter{\romannumeral#1}}%
}

\renewcommand{\tilde}{\widetilde}

\newcommand{\sumtwo}[2]{\sum_{\substack{#1 \\ #2}}} 

\makeatletter
\def\captionfont@{\footnotesize}
\def\captionheadfont@{\scshape}

\long\def\@makecaption#1#2{%
  \vspace{2mm}
  \setbox\@tempboxa\vbox{\color@setgroup
    \advance\hsize-6pc\noindent
    \captionfont@\captionheadfont@#1\@xp\@ifnotempty\@xp
        {\@cdr#2\@nil}{.\captionfont@\upshape\enspace#2}%
    \unskip\kern-6pc\par
    \global\setbox\@ne\lastbox\color@endgroup}%
  \ifhbox\@ne 
    \setbox\@ne\hbox{\unhbox\@ne\unskip\unskip\unpenalty\unkern}%
  \fi
  \ifdim\wd\@tempboxa=\z@ 
    \setbox\@ne\hbox to\columnwidth{\hss\kern-6pc\box\@ne\hss}%
  \else 
    \setbox\@ne\vbox{\unvbox\@tempboxa\parskip\z@skip
        \noindent\unhbox\@ne\advance\hsize-6pc\par}%
\fi
  \ifnum\@tempcnta<64 
    \addvspace\abovecaptionskip
    \moveright 3pc\box\@ne
  \else 
    \moveright 3pc\box\@ne
    \nobreak
    \vskip\belowcaptionskip
  \fi
\relax
}
\makeatother
\def\writefig#1 #2 #3 {\rlap{\kern #1 truecm
\raise #2 truecm \hbox{#3}}}

\title[SOS above a wall in the delocalized phase]{ Typical height of the (2+1)-D Solid-on-Solid surface with pinning above a wall in the delocalized phase}

\author[Naomi Feldheim]{Naomi Feldheim }
 \address{Naomi Feldheim \hfill\break
Bar-Ilan University, 5290002, Ramat Gan, Israel.}
\email{naomi.feldheim@biu.ac.il}

\author[Shangjie Yang]{Shangjie Yang}
 \address{Shangjie Yang \hfill\break
Bar-Ilan University, 5290002, Ramat Gan, Israel.}
\email{shangjie.yang@biu.ac.il}


\begin{document}

\setcounter{tocdepth}{1}


\begin{abstract}
We study the typical height of the (2+1)-dimensional solid-on-solid surface with pinning interacting with an impenetrable wall in the delocalization phase. More precisely, let $\gL_N$ be a $N \times N$ box of $\bbZ^2$, and
we consider a nonnegative integer-valued field $(\phi(x))_{x \in \gL_N}$ with zero boundary conditions (\textit{i.e.} $\phi|_{\gL_N^{\complement}}=0 $) associated with the energy functional 
$$ \cV(\phi)= \gb \sum_{x \sim y} \vert \phi(x)-\phi(y) \vert- \sum_{x} h \ind_{\{ \phi(x)=0\}},$$
where $\beta>0$ is the inverse temperature and $h\ge 0$ is the pinning parameter. 
Lacoin has shown that for  sufficiently large $\gb$, there is a phase transition between delocalization and localization at the critical point 
$$h_w(\gb)= \log \left( \frac{e^{4 \gb}}{e^{4 \gb}-1}\right).$$
In this paper we show that for $\beta\ge 1$ and $h \in (0, h_w)$,
the values of $\phi$
concentrate at the height $H=  \lfloor  (4 \gb)^{-1} \log N \rfloor$ with constant order fluctuations. Moreover, at criticality $h=h_w$, we provide evidence for the conjectured typical height $H_w=  \lfloor  (6 \gb)^{-1} \log N \rfloor$.

\end{abstract}

\maketitle


\section{Introduction}\label{sec:intro}
\subsection{Background}
The solid-on-solid (SOS) model, introduced in \cite{burton1951growth,temperley1952statistical}, is
a crystal surface model, which acts as a qualitative approximation of the Ising model in low temperature (see~\cite{caputo2014dynamics} for more details).

\medskip
Now we formally describe the $(d+1)$-dimensional solid-on-solid model on the lattice  $\bbZ^d$. 
Let $\gL_N \colonequals \lint 1, N \rint^d$ denote a box of size $N$ in the lattice $\bbZ^d$
and  we define its external boundary to be
\begin{equation*}
\partial \gL_N \colonequals  \left \{ x \in \bbZ^d \setminus \gL_N: \ \exists y \in \gL_N, x \sim y  \right\}
\end{equation*}
where $x \sim y $ denotes that $x$ and $y$ are nearest neighbors in the lattice $\bbZ^d$. Given $\phi \in  \tilde \gO_{\gL_N} \colonequals \bbZ^{\gL_N}$, we define the Hamiltonian for the solid-on-solid model  with zero boundary condition as
\begin{equation}\label{def:hamiltonian}
\cH_{N}(\phi) \colonequals \sumtwo{\{x,y\} \subset \gL_N}{x \sim y} \vert \phi(x)-\phi(y)\vert+ \sumtwo{x \in \gL_N,\; y \in \partial \gL_N}{x \sim y} \vert \phi(x) \vert.
\end{equation}
Then  for $\gb>0$ (inverse temperature), we define a probability measure on $ \tilde \gO_N=\bbZ^{\Lambda_N}$ as follows
\begin{equation}\label{def:probclassic}
\forall \phi \in \tilde \gO_N, \quad \quad  \bP_{N}^{\gb}(\phi)  \colonequals \frac{1}{\tilde\cZ_{N}^{\gb}} e^{-\gb \cH_N(\phi)}
\end{equation}
where 
\begin{equation*}
\tilde \cZ_{N}^{\gb} \colonequals \sum_{\psi \in \tilde \gO_N} e^{-\gb \cH_N(\psi)} \le  \left( \frac{1+e^{-d\gb}}{1-e^{-d\gb}}  \right)^{\vert \gL_N \vert}
\end{equation*}
and we refer to \cite[Equations (3.8)-(3.10)]{lacoin2018wetting} for a proof of the last inequality.
It is known (see \cite{fisher1984walks,temperley1952statistical,temperley1956combinatorial})
that for any $\gb>0$, the $(1+1)$-dimensional SOS surface is rough (delocalized),  which means the expectation of  the absolute value of the height at the center diverges in the thermodynamic limit. However, for $d \ge 3$, it is shown in~\cite{bricmont1982surface} by Peierls argument that for any $\gb>0$,  the $(d+1)$-dimensional SOS surface is rigid (localized), that is,
the expectation of  the absolute value of the height at the center is uniformly bounded. The  interesting case is $d=2$ which exhibits a phase transition between rough (for small $\gb$, c.f. \cite{frohlich1981transition, frohlich1981abelian,frohlich1983berezinskii}) and rigid (for large $\gb$, c.f. \cite{brandenberger1982decay,gallavotti1973Some}). Moreover, numerical simulations suggest that $\gb_c \approx 0.806$ is where the delocalization/localization  transition occurs~\cite{caputo2014dynamics}. 

\subsection{The $(2+1)$-dimensional  SOS surface above a wall}\label{subsec:nopin}
 The probability distribution of the $(2+1)$-dimensional SOS interface above an impenetrable wall (taking non-negative integer values) is the conditional distribution
\begin{equation}\label{nopin:spaceprob}
 \forall  \phi \in \gO_N \colonequals \left\{\phi \in \tilde \gO_N: \phi \ge 0 \right\}, \quad \quad \bbP_N^{\gb} \left( \phi \right) \colonequals \bP_N^{\gb}\left(\phi \right)/\bP_N^{\gb}\left( \gO_N \right).
\end{equation} 
In \cite{bricmont1986random},  Bricmont, Mellouki, and Fr\"{o}hlich
   showed that for large $\gb$,  the average height $H$ of the surface
 satisfies $$  \frac{1}{C \gb} \log N \le H \le  \frac{C}{\gb} \log N.$$
 Later in \cite{caputo2014dynamics}, Caputo, Lubetzky,  Martinelli, Sly and Toninelli showed that for  $\gb \ge 1$,  
 the typical height of the surface concentrates at
$$ H = \left\lfloor \frac{1}{4 \gb} \log N \right\rfloor $$
with fluctuations of order $O(1)$, where $\lfloor x \rfloor \colonequals \sup\{n \in \bbZ: \ n \le x \}$, as follows.

\begin{theorema}[{\cite[Theorem 3.1]{caputo2014dynamics}}] \label{th:nonpin}
There exist two universal constants $C,K>0$ such that for all $\gb \ge 1$ and all integer $k \ge K $, we have for all $N$,
\begin{equation*}
\bbP_N^{\gb} \left( \vert\{x \in \gL_N: \ \phi(x) \ge H+k \} \vert  > e^{-2 \gb k}N^2\right)\le e^{ -Ce^{-2 \gb k}N \left(1 \wedge e^{-2 \gb k}N \log^{-8} N \right) },
\end{equation*}
and 
\begin{equation*}
\bbP_N^{\gb} \left( \left\vert\{x \in \gL_N: \ \phi(x) \le H-k \} \right\vert  > e^{-2 \gb k}N^2\right) \le e^{ -e^{\gb k}N}.
\end{equation*}

\end{theorema}

This result describes the effect of the impenetrable wall in the large $\beta$ regime, as the surface is pushed up to the height of order $\frac {1}{4\beta}\log N$, instead of remaining uniformly bounded when no wall is present. This effect is often called \emph{entropic repulsion}.  Furthermore, in  \cite{caputo2016SOSScaling}  these authors provided 
 a full description of the macroscopic shape of the SOS surface, including the scaling limit and fluctuations of the rescaled macroscopic level
lines. 
 In particular, they show in \cite[Theorem 1]{caputo2016SOSScaling} that the surface concentrates on two values: $H$ and $H-1$.
Moreover, Gheissari and  Lubetzky \cite{Gheissari2021entropic} proved a phase transition in the occurrence of entropic repulsion when the hard wall is at a negative level.

\subsection{The $(2+1)$-dimensional SOS surface with pinning above a wall}
In this paper, we are interested in the case where the $(2+1)$-dimensional SOS surface above a wall interacts with a pinning (or wetting) attraction to the wall.

More precisely, we model this surface in the box $\gL_N \subset \bbZ^2$ by an element of $\gO_N= \bbZ_+^{\gL_N}$, where $\bbZ_+\colonequals \bbZ \cap [0, \infty)$.
Given $\gb>0$ and $h \ge 0$, we define the  probability measure  for the $(2+1)$-dimensional SOS surface above a wall  with zero boundary conditions and pinning reward~$h$, namely $\bbP_{N}^{\gb,h}$  on $\gO_N$,  by
\begin{equation}\label{eq:prob-def}
\bbP_{N}^{\gb,h}(\phi)  \colonequals \tfrac{1}{\cZ_N^{\gb,h}} e^{-\gb \cH_{N}(\phi)+ h \left| \{ x \in \gL_N:\ \phi(x)=0\}\right|},
\end{equation}
where 
\begin{equation}\label{partfunct}
\cZ_{N}^{\gb,h} \colonequals \sum_{\phi \in \gO_N} e^{-\gb \cH_N(\phi)+ h \vert \{ x \in \gL: \  \phi(x)=0\} \vert} \le e^{ h \vert \gL_N \vert} \left( \frac{1+e^{-2\gb}}{1-e^{-2\gb}}  \right)^{\vert \gL_N \vert}.
\end{equation}
By \cite[Equation (2.9)]{lacoin2018wetting}, we know the existence of the following limit
\begin{equation*}
\tf(\gb,h) \colonequals \lim_{N \to \infty} \frac{1}{N^2} \log  \cZ_{N}^{\gb,h}
\end{equation*}
which is called the free energy.
By H\"{o}lder's inequality, for $\theta \in [0,1]$ we have 
$$\cZ_N^{\gb, \theta h_1+(1-\theta)h_2} \le \left(\cZ_N^{\gb,h_1}\right)^{\theta} \cdot \left(\cZ_N^{\gb,h_2}\right)^{1-\theta}, $$ 
and then $\tf(\gb,h)$ is increasing and convex in $h$ since $\tf(\gb,h)$ is the limit of a sequence of increasing and convex functions  in $h$.
Therefore, at points where $\tf(\gb,h)$ is differentiable in~$h$, the convexity (c.f. \cite[Appendix A.1.1]{giacomin2007polymer}) allows us to exchange the order of limit and derivative to obtain the asymptotic contact fraction
\begin{equation*}
\partial_h \tf(\gb, h)= \lim_{N \to \infty} \frac{1}{N^2} \bbE_{N}^{\gb,h} \left[ \vert \phi^{-1}(0)\vert \right],
\end{equation*} 
where we have used the notation $\phi^{-1}(A) \colonequals \{x \in \gL_N: \ \phi(x)\in A \}$ for $A \subset \bbZ$, and $\phi^{-1}(k) \colonequals \phi^{-1}(\{ k\})$ for $k \in \bbZ$.
 In  \cite{chalker1982pinning}, Chalker showed that there exists a critical value
\begin{equation}\label{def:hw}
h_w(\gb) \colonequals \sup \left\{h \in \bbR_+: \tf(\gb,h)=\tf(\gb,0) \right\}\end{equation}
which is positive for all $\gb>0$, thus separating the delocalized phase $(\partial_h \tf(\gb,h)=0)$ from the localized phase $(\partial_h \tf(\gb,h)>0)$.
We refer to  the surveys \cite{ioffe2016low, velenik2006localization} for a comprehensive bibliography on the subject of localization/delocalization of surface models.
Chalker further showed that for all $\gb>0$,
\begin{equation}\label{lubd:critval}
\log \left( \frac{e^{4 \gb}}{e^{4 \gb}-1} \right) \le h_w(\gb) \le \log \left( \frac{16(e^{4 \gb}+1)}{e^{4 \gb}-1} \right).
\end{equation}
Later, Alexander, Dunlop and  Miracle-Solé \cite{alexander2011layering} showed that the lower bound in \eqref{lubd:critval} is asymptotically sharp, and  when  $h$ decreases to $h_w$ the system undergoes a sequence of layering transitions (\textit{i.e.} the typical height of the surface varies as $h$ decreases to $h_w$).
More recently, Lacoin proved in~\cite[Proposition 5.1]{lacoin2018wetting}  that for $\gb > \gb_1$ (where $\gb_1 \in (\log 2, \log 3)$ is given by \cite[(2.20)]{lacoin2018wetting}), we have 
\begin{equation}
h_w(\gb)= \log \left( \frac{e^{4 \gb}}{e^{4 \gb}-1} \right),
\end{equation}
and there exists a constant $C_{\gb}$ such that
\begin{equation*}
\forall u \in (0,1], \quad C_\beta^{-1}u^3 \le \tf(\gb, u+h_w(\gb))-\tf(\gb, h_w(\gb))\le C_\beta u^3.
\end{equation*}
In fact, this constant $C_\gb$ can be determined more precisely under additional conditions on $\gb$, for which we refer to \cite[Theorem 2.1]{lacoin2018wetting}. Furthermore, when $h>h_w$,
 a complete picture of the typical height, the Gibbs states and regularity of the free energy is provided in \cite{lacoin2020wetting}.

\subsection{Subcritical regime} In this paper, our goal is to describe the typical height of the $(2+1)$-dimensional SOS surface above a wall with pinning parameter $h \in (0, h_w)$.
Our main result is a generalization of Theorem \ref{th:nonpin} to the subcritical pinning regime. We note that
for $h \in (0, h_w)$ we have $e^{-h}+e^{-4 \gb}>1$, and then define for $\delta>0$,
\begin{equation}\label{def:kappa}
  \kappa(\gb, h, \delta) \colonequals \frac{4 \gb + \delta}{\log \left( e^{-h}+e^{-4 \gb} \right)}. 
\end{equation} 

\begin{theorem}\label{th:subcritical} Fix $\gb \ge 1 $,  $h \in (0, h_w)$ and $N\ge 1$. Let $H = \left\lfloor \frac{1}{4 \gb} \log N \right\rfloor$.

\begin{enumerate}[(i)]

\item\label{upwardspike}
There exist two universal constants $C, K>0$ such that for all integer $m \ge K$,
\begin{equation*}
\bbP_N^{\gb,h} \left(\left\vert \phi^{-1}([H+m, \infty)) \right\vert  > e^{-2 \gb m}N^2\right)
\le e^{ -Ce^{-2 \gb m}N \left(1 \wedge e^{-2 \gb m}N \log^{-8} N \right) }.
\end{equation*}

\item\label{downwardspike} For $\delta>0$ and $m\in \mathbb{N}$ we have
\begin{equation*}
\bbP_N^{\gb,h} \left( \left\vert \phi^{-1}([0, H-m]) \right\vert  >2 e^{-2 \gb m}N^2\right) \le 
 3  e^{ - \min \left( \frac{1}{2}e^{2 \gb m} -4 \gb(1+\gk), \; \delta \right) N },
\end{equation*}
where $\gk$ is defined in \eqref{def:kappa}.
\end{enumerate}
\end{theorem}
We expect that similarly, for $h \in (0,h_w)$, the surface concentrates on two values $H-c(h)$ and $ H-c(h)-1$ where $c(h) \in \bbZ_+$ tends to infinity when $h$ tends to $h_w$.

\subsection{Behavior at criticality}

Next we consider the behavior at critical $h=h_w$,
as defined in~\eqref{def:hw}. Our main result is that 
 the amount of non-isolated zeros is at most of order $N$ with high probability.
 For $\phi \in \gO_N$,  we define its isolated and non-isolated zeros to be respectively
\begin{equation}\label{def:q}
\begin{split}
q_1(\phi) \colonequals \left\{ x \in \gL_N: \ \phi(x)=0, \forall y \in \gL_N, y \sim x, \phi(y) \ge 1\right\},\\
q_{2+}(\phi) \colonequals \left\{ x \in \gL_N: \ \phi(x)=0, \exists y \in \gL_N, y \sim x, \phi(y) =0\right\}.
\end{split}
\end{equation}
We prove the following theorem.
\begin{theorem}\label{th:spikes2}
For $\beta \ge 1$ and $h=h_w$, we have for all $N \in \bbN$ and $C >0$, 
\begin{equation*}
\bbP_N^{\beta, h_w} \left( \phi \in \gO_N: \  \vert q_{2+}(\phi)  \vert \ge CN \right) \le e^{-N\left(\frac{C}{20}e^{-6\beta}-4 \beta \right)}.
\end{equation*}
\end{theorem}

When $h=h_w$, it is conjectured that the surface height concentrates around the value
\begin{equation}\label{heigh:crit}
H_w 
=
\left\lfloor \frac{1}{6 \gb} \log N \right\rfloor,
\end{equation}
with fluctuations similar to Theorem \ref{th:subcritical} (\cite{lacoin2022per}). 
The intuition for this different typical height is a balance at criticality between entropic cost of lifting the surface up and the reward for isolated zeros. 
Theorem~\ref{th:spikes2} indicates that non-isolated zeros should not contribute to this balance.

Our last result gives further evidence for the conjecture. We show that the probability of downwards fluctuations from the conjectured typical height $H_w$ is very small, if the amount of zeros is at most of order $N^{4/3}$.
\begin{proposition}\label{lemacrit:downspike}
 For all $\gb \ge 1$, $C>0$, $h=h_w$, $N\in \bbN$ and $m\in\bbN$, letting $H_w= \lfloor \frac{1}{6 \beta} \log N \rfloor$ we have
\begin{equation*}
\begin{aligned}
\bbP_N^{\gb,h_w} & \left( \left\{ \vert \phi^{-1}(0)\vert \le C N^{\frac{4}{3}} \right\}  \bigcap \left\{ \left\vert \phi^{-1}([1,H_w-m])\right\vert \ge  2 e^{-2 \gb m} N^2 \right\} \right)\\
 &\le  
 2 \exp \left( 4 \gb N +4  \gb  C N^{\frac{4}{3}}-\tfrac{1}{2}e^{2 \gb m}N^{\frac{4}{3}}\right).
\end{aligned}
\end{equation*}
\end{proposition}
As a consequence of  Theorem \ref{th:spikes2} and Proposition \ref{lemacrit:downspike}, it is enough to prove that for large enough $C>0$, we have
\begin{equation}\label{eq: q1-conj}
\bbP_N^{\gb, h_w} \left( \vert q_1(\phi)\vert>CN^{4/3} \right) = o(1), \quad N\to\infty
\end{equation}
in order to obtain a lower bound on the typical height of the surface at criticality, matching the conjectured height in \eqref{heigh:crit}.

\subsection{Open problems and heuristic arguments}
\emph{Subcritical regime.} Theorem~\ref{th:subcritical} opens the door to more advanced questions about the structure of the surface when $h\in (0,h_w)$. Shape results similar to~\cite{caputo2016SOSScaling} are expected to hold.
In particular, we expect that for $h \in (0,h_w)$, the surface concentrates on two values $H-c(h)$ and $ H-c(h)-1$, where $c(h) \in \bbZ_+$ tends to infinity when $h$ tends to $h_w$.

\emph{At criticality.} 
The behavior of the surface at $h=h_w$ remains to be analysed. In particular, proving an analogous result to Theorem~\ref{th:subcritical} with typical height as in~\eqref{heigh:crit} is of interest, as are more advanced shape results (similar to \cite[Theorems 1 and 2]{caputo2016SOSScaling}). As mentioned earlier, our results would imply a lower bound matching this typical height, provided that~\eqref{eq: q1-conj} holds. We turn to offer heuristic arguments for the conjectured estimates~\eqref{heigh:crit} regarding the typical height and ~\eqref{eq: q1-conj} regarding the number of isolated zeros.

For large $\beta$ (low temperature), neighboring vertices tend to have close values. Thus the surface is expected to be nearly flat - most of the vertices will take a certain value~$H_w$, up to the inner boundary. Nonetheless, some vertices will take the value zero (even without pinning, as proved in~\cite{bricmont1986random}).
Due to a spatial mixing property~\cite[Equation (3.13)]{lacoin2018wetting}, these zeros have nearly i.i.d. behavior, with a probability of $e^{-4\beta H_w}$ for each vertex to be a zero (stemming from having $4$ neighbors at height $H_w$, see~\eqref{eq:prob-def}). At criticality, the pinning reward is not influenced at all by isolated zeroes (see~\eqref{def:tP}), and the main contribution is therefore from pairs of neighboring zeros, whose amount is typically $N^2 e^{-6\gb H_w}$. The main penalty to the Hamiltonian~\eqref{def:hamiltonian} is from lifting the inner boundary up to $H_w$ while the outer boundary remains at $0$, thus contributing $e^{-4 \gb N H_w}$ to the computation of probability in~\eqref{eq:prob-def}. 
The value of $H_w$ is one which balances these two factors, that is,
\begin{equation*}
 \exp\left(h_w N^2 e^{-6\gb H_w} \right) e^{-4 \gb N H_w} \asymp 1,
\end{equation*}
which implies $H_w=\lfloor \frac{1}{6 \gb} \log N \rfloor$ as in~\eqref{heigh:crit}. Now, by the i.i.d. model for zeros, 
the mean number of isolated zeros is 
\begin{equation*}
N^2 e^{-4 \gb H_w}=N^{4/3},
\end{equation*}
which implies~\eqref{eq: q1-conj}.

\subsection{Outline of the paper} The paper is organized as follows. Section \ref{sec:thsub-up} is devoted to  Theorem~\ref{th:subcritical}-\eqref{upwardspike} about upward fluctuations in the subcritical regime. 
Section \ref{sec:thsub-dw} is about  Theorem~\ref{th:subcritical}-\eqref{downwardspike} concerning downward fluctuations in the subcritical regime. 
In Section \ref{sec:critical}, we prove Theorem \ref{th:spikes2} and Proposition \ref{lemacrit:downspike} at criticality.

\section{Theorem~\ref{th:subcritical}-(1): upward fluctuations for $h \in (0, h_w)$}\label{sec:thsub-up} 
Intuitively, the height of the $(2+1)$-dimensional SOS surface above a wall with pinning (\textit{i.e.} $h  \ge 0$) is stochastically dominated by that without pinning (\textit{i.e.} $h= 0$). We use this comparison between $\bbP_N^{\gb,h}$ and $\bbP_N^{\gb,0}$ to prove part \eqref{upwardspike},  where $\bbP_N^{\gb,0}=\bbP_N^{\gb}$ is defined in Subsection \ref{subsec:nopin}.

\subsection{Partial order and stochastic domination} We define a partial order "$\le $" on $\gO_N \times \gO_N$ as follows
\begin{equation*}
\phi \le \psi  \quad \Leftrightarrow  \quad \forall x \in \gL_N, \text{ } \phi(x) \le \psi(x).
\end{equation*}
Moreover, a function $f: \gO_N \mapsto \bbR$ is increasing if 
\begin{equation*}
\phi \le \psi  \quad \Rightarrow \quad  f(\phi) \le f(\psi). 
\end{equation*}
Similarly, an event $\cA \subset \gO_N$ is increasing if its indicator function $\ind_{\cA}$ is increasing.
For two probability measures $\mu_1, \mu_2$ on $\gO_N$, we say that $\mu_2$ dominates $\mu_1$, denoted by $\mu_1 \preceq \mu_2$, if for any bounded increasing function $f: \gO_N \mapsto \bbR$, we have
\begin{equation*}
\mu_1(f) \le \mu_2(f).
\end{equation*}

\begin{lemma} \label{lema:stochdom}
For all $\gb>0$ and $ 0 \le h_1 \le h_2$,  we have
\begin{equation}\label{unpindom}
\bbP_N^{\gb,h_2} \preceq \bbP_N^{\gb,h_1}.
\end{equation}
\end{lemma}

\begin{proof}
Since \cite[Theorem 6]{holley1974remarks} is applied for finite distributive lattice, we set $$\cA_n\colonequals \left\{\phi \in \gO_N: \max_{x \in \gL_N} \phi(x) \le n \right\}.$$
 It is fundamental to  verify Holley's condition \cite[Equation (7)]{holley1974remarks} to obtain
\begin{equation*}
\bbP_N^{\gb,h_2} \left(\cdot \ \vert \  \cA_n\right) \preceq \bbP_N^{\gb,h_1} \left(  \cdot \ \vert \ \cA_n \right),
\end{equation*}
and then for any bounded increasing function $f: \gO_N \mapsto \bbR$ we have
\begin{equation} \label{stochdom}
 \frac{\bbE_N^{\gb,h_2}[f \ind_{\cA_n}]}{\bbP_N^{\gb,h_2}(\cA_n)} \le \frac{ \bbE_N^{\gb,h_1}[f \ind_{\cA_n}]}{\bbP_N^{\gb,h_1}(\cA_n)}. 
\end{equation}
Moreover, by the dominate convergence theorem, for all $h \ge 0$ we have 
\begin{equation}\label{dom:approx}
\bbE_N^{\gb,h} \left[ f \right]
= \lim_{n \to \infty} \frac{\bbE_N^{\gb,h}[f \ind_{\cA_n}]}{\bbP_N^{\gb,h}(\cA_n)}.
\end{equation}
Combining \eqref{stochdom} and \eqref{dom:approx}, we conclude the proof.
\end{proof}

\subsection{Proof of Theorem ~\ref{th:subcritical}-\eqref{upwardspike}. } %
Note that for any integer $m$,
 the event $$\left\{ \phi \in \gO_N: \ \vert\{x \in \gL_N: \ \phi(x) \ge H+m \} \vert  > e^{-2 \gb m}N^2 \right\}$$
is increasing. We combine Lemma \ref{lema:stochdom} and Theorem \ref{th:nonpin} to conclude the proof.
\qed

\section{Theorem~\ref{th:subcritical}-(ii): downward fluctuations for $h \in (0, h_w)$}\label{sec:thsub-dw} 

To prove part~\eqref{downwardspike} of Theorem~\ref{th:subcritical},
 we first show that $\vert \phi^{-1}(0)\vert$ is at most of order $N$, with high probability, adopting the strategy in \cite[Theorem 3.1]{caputo2014dynamics}. 

\begin{lemma} \label{lema:pinorder}
For all $\gb \ge 1$, $h \in [0, h_w)$, $\delta>0$ and $N \ge 1$, we have
\begin{equation*}
\bbP_N^{\gb,h} \left( \vert \phi^{-1}(0)\vert  \ge \gk N \right) \le e^{-\delta N},
\end{equation*}
where $\gk= \kappa(\gb, h, \delta)$ is defined in \eqref{def:kappa}.
\end{lemma}

\begin{proof}
For $\phi \in \gO_N$ and each $A \subseteq  \phi^{-1}(0)$, we define  $U_A \phi: \gL_N \mapsto \bbZ_+$ as follows
\begin{equation*}
(U_A \phi)(x) \colonequals
\begin{cases}
\phi(x)+1, & \text{ if } x \not\in A,\\
0, & \text{ if } x \in A. 
\end{cases}
\end{equation*}
Since the action $U_A$ increases the height of each site in $\gL_N \setminus A$ by one, we have
\begin{equation}
\begin{gathered}
\cH_N(U_A \phi) \le \cH_N(\phi)+4 \vert A \vert+4N,\\
\vert \phi^{-1}(0) \vert- \vert (U_A \phi)^{-1}(0) \vert= \vert \phi^{-1}(0) \setminus A\vert.
\end{gathered}
\end{equation} 
Therefore,
\begin{equation*}
\bbP_N^{\gb,h} \left( U_A \phi \right) \ge \bbP_N^{\gb,h} \left( \phi \right) \cdot \exp \left( -h \vert  \phi^{-1}(0) \setminus A\vert- 4 \gb \vert A \vert -4\gb N  \right),
\end{equation*}
and then
\begin{equation}
\begin{aligned} 
\label{entsum}
\sum_{A \subseteq \phi^{-1}(0)} \bbP_N^{\gb,h} \left( U_A \phi \right) & \ge e^{-4\gb N } \cdot \bbP_N^{\gb,h} \left( \phi \right)  \sum_{A \subseteq \phi^{-1}(0)} \exp \left( -h \vert \phi^{-1}(0) \setminus A\vert- 4 \gb \vert A \vert \right) \\
& =e^{-4 \gb N - h \vert \phi^{-1}(0) \vert} \cdot \bbP_N^{\gb,h} \left( \phi \right) \sum_{n=0}^{\vert \phi^{-1}(0) \vert}   \sumtwo{A \subseteq \phi^{-1}(0)}{\vert A \vert =n} \exp\left( -n(4 \gb-h ) \right)\\
& =e^{-4 \gb N - h \vert \phi^{-1}(0)\vert}  \left(1+e^{-(4\gb-h)}\right)^{\vert \phi^{-1}
(0) \vert} \bbP_N^{\gb,h} \left( \phi \right).
\end{aligned} 
\end{equation}
 Observe that for $A, A' \subseteq \phi^{-1}(0)$ with $A \neq A'$, we have
$$ U_A \phi \neq U_{A'}\phi .$$
Furthermore, for $\phi \neq \psi$,  if $A \subseteq \phi^{-1}(0)$ and $B \subseteq \psi^{-1}(0)$, we have
\begin{equation*}
U_A \phi \neq U_B \psi,
\end{equation*}
because we can recover $A$ from $U_A \phi$ by zero-value sites and then proceed to recover $\phi$. Therefore
$\sum_{\phi \in \Omega_N} \ \sum_{A \subset \phi^{-1}(0)} \bbP_N^{\gb,h}(U_A \phi) \le 1$.
In particular, using \eqref{entsum} we obtain
\begin{equation}
\begin{aligned} \label{upbd:pinlinear}
1 &\ge \sum_{\phi: \ \vert \phi^{-1}(0)\vert \ge \gk N} \ \sum_{A \subset \phi^{-1}(0)} \bbP_N^{\gb,h}(U_A \phi)\\
&\ge  \sum_{\phi: \ \vert \phi^{-1}(0) \vert \ge \gk N} e^{-4\gb N - h \vert \phi^{-1}(0) \vert}  \left(1+e^{-(4\gb-h)}\right)^{\vert \phi^{-1}(0) \vert} \bbP_N^{\gb,h} \left( \phi \right)\\
&\ge  e^{-4 \gb N}  \left(e^{-h}+e^{-4\gb}\right)^{ \gk N} \bbP_N^{\gb,h} \left(   \vert \phi^{-1}(0) \vert \ge \gk N \right)
\end{aligned} 
\end{equation}
where in the last inequality we have used that $e^{-h}+e^{-4\gb} > 1$ for $h \in [0, h_w)$. 
By the definition of $\kappa$ in \eqref{def:kappa} we have
 $$ \left( e^{-h}+e^{-4\gb} \right)^{\gk} e^{-4 \gb}=e^{\delta}.$$
Plugging this into \eqref{upbd:pinlinear}, we conclude the proof of Lemma~\ref{lema:pinorder}.
\end{proof}

\begin{lemma}\label{lema:downspike}
 Let $\gb \ge 1$, $h \in [0,h_w)$ and $\kappa>0$. Then for all $m > \lceil \frac{1}{2 \gb} \log \left( 8 \gb(1+\gk) \right) \rceil$ and $N \ge 1$ we have

\begin{align*}
&\bbP_N^{\gb,h} \left( \Big\{ \vert \phi^{-1}(0)\vert \le \gk N \Big\}  \bigcap \left\{ \left\vert \phi^{-1}([1,H-m])\right\vert \ge \frac{ e^{-2 \gb m}}{1-e^{-2 \gb}} N^2 \right\} \right)\\
 \le &\frac{1}{1-e^{-\gb N}} e^{ - \left( \frac{1}{2}e^{2 \gb m} -4 \gb(1+\gk)\right) N }.
\end{align*}
\end{lemma}

\begin{rem}
 The condition $m >\frac{1}{2 \gb} \log \left( 8 \gb(1+\gk) \right)$ is only to ensure that $\tfrac{1}{2}e^{2 \gb m}-4 \gb(1+\gk)>0$.
\end{rem}

\begin{proof}
Fix an integer $\ell \in [ 1, H-m ]$. For any subset $A \subseteq \phi^{-1}(\ell)$, we define  $V_A \phi: \gL_N \mapsto \bbZ_+$ as follows
\begin{equation}\label{mapphi}
(V_A \phi)(x) \colonequals
\begin{cases}
0, & \text{ if } x \in \phi^{-1}(0),\\
1, & \text{ if } x \in A,\\
\phi(x)+1, & \text{ if } x \not\in A \cup \phi^{-1}(0).
\end{cases}
\end{equation}
Observe that for $x \in A$ and $y \not \in A \cup \phi^{-1}(0)$ with $x \sim y$,
$$ \vert (V_A \phi)(x)- (V_A \phi)(y) \vert= \phi(y) \le \vert \ell- \phi(y)\vert+\ell, $$
and then
$$ \cH_N(V_A \phi) \le \cH_N(\phi)+4N+4 \vert \phi^{-1}(0)\vert+4 \ell \vert A \vert.  $$
Moreover, as
$\vert (V_A \phi)^{-1}(0)\vert= \vert \phi^{-1}(0)\vert$, we obtain
\begin{equation*}
\bbP_N^{\gb,h}\left( V_A \phi \right) \ge \bbP_N^{\gb,h}\left( \phi \right) e^{-4\gb N- 4\gb \vert  \phi^{-1}(0)\vert-4  \gb \ell \vert A \vert}.
\end{equation*}
Similarly to \eqref{entsum}, we have
\begin{equation}\label{sumsubsetV}
\begin{aligned}
    \sum_{A \subseteq \phi^{-1}(\ell)} \bbP_N^{\gb,h}\left( V_A \phi \right) &\ge \bbP_N^{\gb,h}\left( \phi \right) \sum_{A \subseteq \phi^{-1}(\ell) }  e^{-4\gb N- 4\gb \vert  \phi^{-1}(0)\vert-4 \gb \ell  \vert A \vert}\\
&= \bbP_N^{\gb,h}\left( \phi \right)  e^{-4\gb N- 4\gb \vert  \phi^{-1}(0)\vert} \left( 1+e^{-4 \gb \ell }\right)^{\vert \phi^{-1}(\ell)\vert}\\
&\ge  
\bbP_N^{\gb,h}\left( \phi \right)  \exp \left( -4\gb N- 4\gb \vert  \phi^{-1}(0)\vert + \tfrac{1}{2}e^{-4 \gb \ell}\vert \phi^{-1}(\ell)\vert \right)
\end{aligned}
\end{equation}
where  we have used $(1+x) \ge e^{x/2}$ for $x \in [0,1]$ in the last inequality.

Note that for $A,A' \subseteq \phi^{-1}(\ell)$ with $A \neq A'$, we have $$V_A \phi \neq V_{A'}\phi.$$
Moreover, for $\phi \neq \psi \in \gO_N$, $A \subset \phi^{-1}(\ell)$ and $B \subset \psi^{-1}(\ell)$, we have
$$ V_A \phi \neq V_B\psi, $$
since we can recover $A$  by $1-$valued sites of $V_A \phi$ and then proceed to recover $\phi$.
Therefore,  by~\eqref{sumsubsetV}, denoting $j=H-\ell$ we obtain

\begin{equation}\label{upbd:downspike}
    \begin{aligned}
1 &\ge \sumtwo{\phi: \; \vert \phi^{-1}(\ell) \vert \ge e^{-2 \gb j}N^2}{\vert \phi^{-1}(0) \vert \le \gk N} 
\ \sum_{A \subset \phi^{-1}(\ell)}
\bbP_N^{\gb, h} \left( V_A \phi \right) \\
&\ge \exp \left( -4 \gb N  -4  \gb  \gk N+\tfrac{1}{2}e^{2 \gb j}N \right)
\bbP_N^{\gb,h} \left( \{ \vert \phi^{-1}(\ell) \vert \ge e^{-2 \gb j}N^2 \} \cap  \{ \vert \phi^{-1}(0) \vert \le \gk N\} \right).
\end{aligned}
\end{equation}
Moreover, as
$$ \left\{ \vert  \phi^{-1}([1,H-m])\vert \ge \frac{ e^{-2 \gb m}}{1-e^{-2 \gb}}N^2 \right\} \subset \bigcup_{j=m}^{H-1} \left\{ \vert  \phi^{-1}(H-j)\vert \ge  e^{-2 \gb j}N^2 \right\}, $$
by union bound and \eqref{upbd:downspike} we obtain
    \begin{align*}
\bbP_N^{\gb,h} 
& \left( \left\{ \vert \phi^{-1}([1,H-m]) \vert \ge \frac{ e^{-2 \gb m}}{1-e^{-2 \gb}} N^2 \right\} \bigcap  \Big\{ \vert \phi^{-1}(0) \vert \le  \gk N \Big\} \right)  \\
 \le & \sum_{j=m}^{H-1} \bbP_N^{\gb,h} \left( \{ \vert \phi^{-1}(H-j) \vert \ge  e^{-2 \gb j} N^2 \} \bigcap  \Big\{ \vert \phi^{-1}(0) \vert \le  \gk N \Big\}\right) \\
\le &\sum_{j=m}^{H-1} \exp \left( 4 \gb N +4  \gb  \gk N-\tfrac{1}{2}e^{2 \gb j}N \right)\\
\le & \frac{1}{1-e^{-\gb N}} \exp \left( 4 \gb N +4  \gb  \gk N-\tfrac{1}{2}e^{2 \gb m}N \right),
\end{align*}
where in the last inequality we have used that for $j \ge 0$,
$$ \frac{  \exp \left(-\frac{1}{2} e^{2 \gb (j+1)}N \right)  }{ \exp \left( -\frac{1}{2} e^{2 \gb j}N\right) }  \le \exp \left(-\gb e^{2 \gb j}N \right) \le e^{-\gb N}.$$
This concludes the proof. \qedhere
\end{proof}

\subsection{Proof of 
 Theorem ~\ref{th:subcritical}-\eqref{downwardspike}}
 For for  all $N \ge 1$, we have
\begin{align*}
\bbP_N^{\gb,h} &\left( \left\vert \phi^{-1}([0,H-m]) \right\vert  >  2e^{-2 \gb m}N^2\right) \\
& \le \bbP_N^{\gb,h} \left( \left\vert \phi^{-1}(0) \right\vert  > \gk N\right)+  \\
&\quad \bbP_N^{\gb,h} \left( \Big\{ \vert \phi^{-1}(0)\vert \le \gk N \Big\}  \bigcap \left\{ \left\vert \phi^{-1}([1,H-m])\right\vert \ge \frac{ e^{-2 \gb m}}{1-e^{-2 \gb}}N^2 \right\} \right)\\
&\le e^{-\delta N}+\frac{1}{1-e^{-\gb N}} \exp \left( - \left( \tfrac{1}{2}e^{2 \gb m} -4 \gb(1+\gk)\right) N \right)\\
&\le 3  \exp \left( - \min \left( \tfrac{1}{2}e^{2 \gb m} -4 \gb(1+\gk),\delta \right) N \right),
\end{align*}
where we have applied Lemma \ref{lema:pinorder} and  Lemma \ref{lema:downspike} in the second inequality.
\qed

\section{Theorem \ref{th:spikes2}:  upper bound on non-isolated zeros at criticality}\label{sec:critical}
 This section is devoted to the proof of Theorem \ref{th:spikes2}.
  Inspired by \cite[Lemma 3.1]{lacoin2018wetting}, we first observe that for $x_1,x_2,x_3,x_4 \in \bbZ_+$,
\begin{equation}\label{replace:isospikes}
\sum_{k=- \infty}^{0} \exp \left( - \beta \sum_{i=1}^4 \vert x_i-k \vert \right)= \exp\left( h_w- \beta\sum_{i=1}^4 x_i\right).
\end{equation}
Define a new state space
\begin{equation}\label{newspace}
\gO_N^* \colonequals  \left\{\psi: \ \gL_N \to \bbZ \ | \text{ if } \psi(x) \le -1, \forall y \in \gL_N, y \sim x, \psi(y) \ge 1 \right\}.
\end{equation}
Notice that if $\psi \in\gO_N^*$, then $\max(\psi,0)\in \gO_N$ (as defined in \eqref{nopin:spaceprob}).
By \eqref{replace:isospikes}, we have
\begin{equation*}
\cZ_N^{\beta, h_w}=\sum_{\psi \in \gO_N^*} \exp \left(-\beta \cH_N(\psi)+h_w \vert  q_{2+}(\psi)\vert  \right)
\end{equation*}
where $\cZ_N^{\beta, h_w}$ is defined in \eqref{partfunct} and $\cH_N$ is defined in \eqref{def:hamiltonian} with zero boundary condition.
Define a new probability measure
 $ \tP_N$ on $\gO_N^*$ as follows:
\begin{equation}\label{def:tP}
\forall \psi \in \gO_N^*, \quad \tP_N(\psi) \colonequals \frac{1}{\cZ_N^{\beta,h_w}} \exp\left( -\beta \cH_N(\psi)+h_w \vert q_{2+}(\psi) \vert  \right).
\end{equation}
Observation~\eqref{replace:isospikes} yields the following relation between $\tP_N$ and $\bbP_N^{\beta, h_w}$: for any $\phi\in\gO_N$, 
\begin{equation*}
\bbP_N^{\beta, h_w}(\phi)=\tP_N \left( \left\{ \psi \in \gO_N^*: \max(\psi,0)=\phi  \right\} \right).
\end{equation*}
 In particular, we have
\begin{equation}\label{relate:q2}
\bbP_N^{\beta, h_w} \left( \left\{ \phi \in \gO_N: \ \vert q_{2+}(\phi) \vert  \ge C N \right\} \right)
=\tP_N \left( \left\{ \psi \in \gO_N^*:  \ \vert q_{2+}(\psi) \vert  \ge C N \right\} \right),
\end{equation}
since for any $\psi\in \gO_N^*$, we have $q_{2+}(\max(\psi,0)) = q_{2+}(\psi)$.

From now on, we deal with the r.h.s. of \eqref{relate:q2}.
For any subset $A \subseteq q_{2+}(\psi)$, we let $\cN(A)$ be the edge boundary of $A$, defined by
\begin{equation}\label{def:Na}
\cN(A) \colonequals \left\{ \{ x,y\} \in E(\bbZ^2): \ x \in A, y \in A^{\complement} \right\}, \end{equation}
and define $U_A \psi \in \gO_N^*$ as
\begin{equation*}
(U_A \psi)(x)\colonequals
\begin{cases}
\psi(x)+1 & \text{ if }  x \not\in A,\\
0 & \text{ if } x \in A.
\end{cases}
\end{equation*}
For ease of notation, we fix $\psi\in \gO_N^*$ and write $q_{2+}(A) \colonequals q_{2+}(U_A \psi)$ in the sequel.
Observing 
$\cH_N(U_A \psi) \le \cH_N(\psi)+4\beta N+ \beta \vert \cN(A) \vert$,
we have by \eqref{def:tP}:
\begin{equation*}
\tP_N(U_A \psi) \ge \tP_N( \psi) \exp \left( -4 \beta N-  \beta \vert \cN(A) \vert - h_w\left( \vert q_{2+}(\psi) \vert - \vert q_{2+}(A) \vert\right)  \right).
\end{equation*}
Let $V_1, V_2, \cdots, V_k$ be the connected components of $q_{2+}(\psi)$, and write $A_i= A \cap V_i$. We sum over all subsets $A \subseteq q_{2+}(\psi)$ to obtain
\begin{equation}\label{comput:component}
\begin{aligned}
\sum_{A \subseteq q_{2+}(\psi)}&\tP_N(U_A \psi) \\
&\ge 
\tP_N( \psi) \exp \left(-4 \beta N- h_w \vert q_{2+}(\psi) \vert  \right)  \sum_{A \subseteq q_{2+}(\psi)} \exp \left( -  \beta \vert \cN(A) \vert + h_w\vert q_{2+}(A) \vert  \right)\\
& =\tP_N( \psi) \exp \left(-4 \beta N- h_w \vert q_{2+}(\psi) \vert  \right)  \sum_{A_1, \cdots, A_k}  \prod_{i=1}^k \exp \left( -  \beta \vert \cN(A_i) \vert + h_w\vert q_{2+}(A_i) \vert  \right)\\
& =\tP_N( \psi) \exp \left(-4 \beta N \right)   \prod_{i=1}^k  \exp \left(- h_w \vert V_i \vert\right) \sum_{A_i \subseteq V_i }  \exp \left( -  \beta \vert \cN(A_i) \vert + h_w\vert q_{2+}(A_i) \vert  \right)
\end{aligned}
\end{equation}
where 
we have used that $q_{2+}(\psi)=  V_1 \cup V_2\cdots \cup  V_k$ is the disjoint union of $V_1, V_2, \cdots, V_k$.

Note that for any finite connected subgraph of $\bbZ^2$, after deleting some edges (but keeping all the vertices), the graph can be decomposed into 
a disjoint union of patterns from Figure \ref{fig:4patterns} (up to rotation and reflection).
 From now on, we focus on one connected component $V_i$ in the r.h.s. of \eqref{comput:component}. 
Denote by $E_i$ the set of edges in this disjoint union of patterns. 
In the graph $(A_i,E_i)$ we define the count of non-isolated spikes (similar to~\eqref{def:q}) by
\[
\widetilde{q}_{2+}(A_i) \colonequals \left\{ x \in A_i: \ \psi(x)=0, \exists y \in A_i, (x,y)\in E_i, \psi(y) =0\right\},
\]
and the edge boundary, similar to~\eqref{def:Na}, by
\[
\widetilde{\cN}(A_i) =  \left\{ \{x,y\}\not\in E_i:  x\in A_i \lor y\in A_i  \right\}.
\]
Observe that 
\begin{equation}\label{eq:tilde}
| \tilde{q}_{2+}(A_i)| \le |q_{2+}(A_i)|, \text{  and   }|\tilde{\cN}(A_i)| \ge |\cN(A_i)|.
\end{equation}
The next lemma will therefore provide a lower bound on the r.h.s. in \eqref{comput:component}.

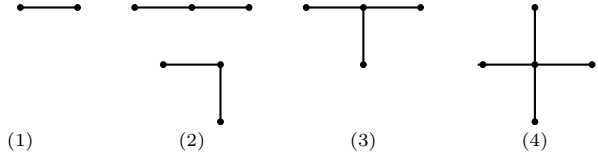
\begin{figure}[h]
 \centering
   \begin{tikzpicture}[scale=.38,font=\tiny]
   
    \draw[thick] (0,0) -- (2,0);
     \foreach \x in {0,2} {\draw[fill] (\x,0) circle [radius=0.1];}

     \node[below] at (0,-4) {$(1)$};

  \draw[thick]  (4,0)--(8,0);
  \foreach \x in {4,6,8} {\draw[fill] (\x,0) circle [radius=0.1];}

 \draw[thick]  (4+1,-2)--(6+1,-2)--(6+1,-4);
  \foreach \x in {4+1,6+1} {\draw[fill] (\x,-2) circle [radius=0.1];}
  \draw[fill] (7,-4) circle [radius=0.1];
\node[below] at (6,-4) {$(2)$};

 \draw[thick]  (10,0)--(14,0);
  \draw[thick]  (12,0)--(12,-2);

  \foreach \x in {10,12,14} {\draw[fill] (\x,0) circle [radius=0.1];}
  \draw[fill] (12,-2) circle [radius=0.1];
  \node[below] at (12,-4) {$(3)$};
  
  \draw[thick]  (16,-2)--(20,-2);
    \foreach \x in {16.18,20} {\draw[fill] (\x,-2) circle [radius=0.1];}

  \draw[thick]  (18,0)--(18,-4);
      \foreach \x in {0,-2,-4} {\draw[fill] (18,\x) circle [radius=0.1];}
  \node[below] at (18,-4) {$(4)$};

   \end{tikzpicture}
 
  \caption{The vertices of any connected set of size at least $2$ can be covered by a disjoint union of these four patterns. (The two configurations in $(2)$ are considered as the same pattern, as both have $3$ vertices and  $8$ boundary edges.)
   }\label{fig:4patterns}
 \end{figure}

\begin{lemma} \label{lema:smalllwbd}
If $\gb \ge 1$  and $V$ is the vertex set of one of the patterns shown in Figure \ref{fig:4patterns}, then
\begin{equation}\label{lwbd:smallset}
\exp \left(-h_w \vert V \vert \right) \sum_{B \subseteq V} \exp \left(-\beta \vert \cN(B) \vert+h_w \vert q_{2+}(B)\vert \right) \ge 1+\frac{1}{2}e^{-6 \beta}.
\end{equation}
\end{lemma}

\begin{proof}
For simplicity of notation, we write $h=h_w$. We will repeatedly use that
$e^{-h}=1-e^{-4\beta}$ and $\beta \ge 1$ without further reference.
We consider the four patterns in Figure \ref{fig:4patterns}, case by case. 
If $\vert V\vert=2$ (\textit{i.e.} (1) in Figure \ref{fig:4patterns}), then the l.h.s. of \eqref{lwbd:smallset} equals to 
\begin{multline*}
e^{-2h} \left(1+ 2e^{-4 \beta}+e^{-6 \beta+2h} \right) =
(1-e^{-4\gb})^2 (1+2 e^{-4 \gb})+e^{-6 \gb} 
\ge 
1+e^{-6 \beta}(1-3e^{-2\beta})>1+\frac{1}{2}e^{-6 \beta}.
\end{multline*}

In the case $\vert V \vert=3$ (pattern (2) in Figure \ref{fig:4patterns}), the l.h.s. of \eqref{lwbd:smallset} equals to 
\begin{multline*}
e^{-3 h} \left( 1+ 3e^{-4 \beta}+e^{-8 \beta}+2e^{-6 \beta+2h}+e^{-8 \beta+3h}\right) 
\ge (1-e^{-4\gb})^3 \left(1+3 e^{-4 \beta} \right)+2e^{-6 \beta-h} \\
\ge 1+2 e^{-6 \beta} \left(1-3e^{-2 \beta}-e^{-4\gb} \right) \ge 1+e^{-6 \beta}.
\end{multline*}

Consider now the case $\vert V \vert=4$, corresponding to pattern (3) in Figure \ref{fig:4patterns}.   By counting connected subsets of size at most two,  the l.h.s. of \eqref{lwbd:smallset} is bounded from below by
\begin{equation*}
\begin{aligned}
e^{-4h} \left(1+ 4 e^{-4 \beta}+3e^{-6 \beta+2h}\right)
&\ge (1-e^{-4\gb})^4\left(1+4e^{-4 \beta} \right)+3 e^{-6 \beta-2h}\\
&\ge 1-10 e^{-8 \gb}+3 e^{-6 \beta-2h}
\ge 1+e^{-6 \beta}.
\end{aligned}
\end{equation*}

Lastly, consider the case $\vert V\vert=5$, corresponding to pattern (4) in Figure \ref{fig:4patterns}.  By counting connected subsets of size at most two,  the l.h.s. of \eqref{lwbd:smallset} is larger than 
\begin{equation*}
\begin{aligned}
e^{-5h} \left( 1+5 e^{-4 \beta}+4e^{-6\beta+2h} \right)
&\ge  1-15e^{-8 \gb}+ 4e^{-6\beta+2h}(1-5e^{-4\gb})\\
&\ge 1+e^{-6 \beta} \left(4 \frac{1-5e^{-4\gb}}{(1-e^{-4\gb})^2}-15 e^{-2\gb} \right)
 \ge 1+e^{-6 \beta}.
\end{aligned}
\end{equation*}
This concludes the proof.
\end{proof}

With Lemma \ref{lema:smalllwbd} at hand, we are ready to prove Theorem \ref{th:spikes2}.

\begin{proof}[Proof of Theorem  \ref{th:spikes2}]
Using~\eqref{eq:tilde} and~\eqref{lwbd:smallset}, we may continue the inequality~\eqref{comput:component} to obtain
\begin{equation}\label{lwbd:decomposed}
\sum_{A \subseteq q_{2+}(\psi)}\tP_N(U_A \psi) \ge
e^{-4 \beta N} \tP_N(\psi) \left(1+ \frac{1}{2}e^{-6 \beta} \right)^{\vert q_{2+}(\psi)\vert/5},
\end{equation}
 where we have used that the total numbers of patterns covering $q_{2+}(\psi)$ is at least $\vert q_{2+}(\psi)\vert/5 $. 
Note that for $A \neq B  \subset q_{2+}(\psi)$ we have $U_A \psi \neq U_{B} \psi$ since $(U_A \psi) |_{A \setminus B}=0$ and $(U_{B} \psi) |_{A \setminus B}=1$.
Moreover, for $\psi, \psi' \in \gO_N^*$ and $A \subset q_{2+}(\psi), A' \subset q_{2+}(\psi')$, we have
$$ U_A \psi \neq U_{A'} \psi'.$$
To see this, note that 
\[
A = \{x\in \gL_N: \: (U_A\psi)(x)=0, \: \exists y\in \gL_N, \: y \sim x, \: (U_A\psi)(y)\in\{0,1\} \}.
\]
Thus, given $U_A \psi$, we can first recover the set $A$ and then proceed to recover~$\psi$.
 Therefore, from~\eqref{lwbd:decomposed} we obtain
\begin{equation*}
\begin{aligned}
1 &\ge \sum_{\psi \in \gO_N^*: \; \vert q_{2+}(\psi) \vert \ge CN } \; \sum_{A \subset q_{2+}(\psi)} \tP_N \left( U_A \psi \right)  \\
&\ge  \sum_{\psi \in \gO_N^*:\;  \vert q_{2+}(\psi) \vert \ge CN } e^{-4\beta N} \, \tP_N (\psi) \left( 1+ \tfrac{1}{2} e^{-6 \beta}\right)^{\vert q_{2+}(\psi)\vert /5}\\
&\ge e^{-4 \beta N} \left(1+ \tfrac{1}{2}e^{-6\beta} \right)^{CN/5} \tP_N \left(\left\{\psi \in \gO_N^*: \; \vert  q_{2+}(\psi)\vert \ge CN \right\}\right)\\
&\ge e^{-4 \beta N} e^{\frac{C}{20}e^{-6 \beta} N} \; \tP_N \left(\{\psi \in \gO_N^*:\; \vert  q_{2+}(\psi)\vert \ge CN\} \right),
\end{aligned}
\end{equation*}
where the last step used the inequality $1+x \ge e^{\frac{1}{2}x}$ for $x \in [0,1].$
By \eqref{relate:q2}, this concludes the proof.
\end{proof}

Now we move to prove Proposition \ref{lemacrit:downspike}.
\begin{proof}[Proof of Proposition \ref{lemacrit:downspike}]
For $\ell \in \lint 1, H_w-m \rint$, with exactly the same argument as in \eqref{sumsubsetV} and \eqref{upbd:downspike}, 
setting  $m=H_w-\ell$ we have
\begin{multline}\label{upbd:critdownward}
1 \ge \sumtwo{\psi: \; \vert \psi^{-1}(\ell) \vert \ge e^{-2 \gb m}N^2}{\vert \psi^{-1}(0) \vert \le C N^{\frac{4}{3}}} 
\ \sum_{A \subset \phi^{-1}(\ell)}
\bbP_N^{\gb, h_w} \left( V_A \phi \right) \\
\ge \exp \left( -4 \gb N  -4  \gb  C N^{\frac{4}{3}}+\frac{1}{2}e^{2 \gb m}N^{\frac{4}{3}} \right)
\bbP_N^{\gb,h_w} \left( \{ \vert \phi^{-1}(\ell) \vert \ge e^{-2 \gb m}N^2 \} \cap  \{ \vert \phi^{-1}(0) \vert \le C N^{\frac{4}{3}}\} \right)
\end{multline}
where $V_A \phi$ is defined in \eqref{mapphi}.
Moreover, as $(1-e^{-2 \gb})^{-1} \le 2$ and
$$ \left\{ \vert  \phi^{-1}([1,H_w-m])\vert \ge \frac{ e^{-2 \gb m}}{1-e^{-2 \gb}}N^2 \right\} \subset \bigcup_{i=m}^{H_w-1} \left\{ \vert  \phi^{-1}(H_w-i)\vert \ge  e^{-2 \gb i}N^2 \right\}, $$
by union bound and \eqref{upbd:critdownward} we obtain
\begin{equation*}
\begin{aligned}
&\bbP_N^{\gb,h_w} \left( \left\{ \vert \phi^{-1}([1,H_w-m]) \vert \ge \frac{ e^{-2 \gb m}}{1-e^{-2 \gb}} N^2 \right\} \bigcap  \Big\{ \vert \phi^{-1}(0) \vert \le  C N^{\frac{4}{3}} \Big\} \right) \\
\le &\sum_{i=m}^{H_w-1} \exp \left( 4 \gb N +4  \gb  C N^{\frac{4}{3}}-\frac{1}{2}e^{2 \gb i}N^{\frac{4}{3}} \right)\\
\le &\frac{1}{1-e^{-\gb N^{\frac{4}{3}}}} \exp \left( 4 \gb N +4  \gb  C N^{\frac{4}{3}}-\frac{1}{2}e^{2 \gb m}N^{\frac{4}{3}} \right),
\end{aligned}
\end{equation*}
where in the last inequality we have used that for $j \ge 0$,
$$ \frac{  \exp \left(-\frac{1}{2} e^{2 \gb (j+1)}N^{\frac{4}{3}} \right)  }{ \exp \left( -\frac{1}{2} e^{2 \gb j}N^{\frac{4}{3}} \right) }  \le \exp \left(-\gb e^{2 \gb j}N^{\frac{4}{3}} \right) \le e^{-\gb N^{\frac{4}{3}}}.$$
\end{proof}
 
\subsection*{Acknowledgments}
We are grateful to Hubert Lacoin for suggesting the problem, thank  Ohad Feldheim, Hubert Lacoin and Ron Peled for enlightening discussions,  and thank  Tom Hutchcroft and Fabio Martinelli for pointing out the references \cite{Gheissari2021entropic} and \cite{caputo2016SOSScaling} respectively. 
N.F. is supported by Israel Science Foundation grant 1327/19.  S.Y. is  supported by the Israel Science Foundation grants 1327/19 and 957/20.
This work was partially performed  when S.Y. was  visiting  IMPA.

\bibliographystyle{abbrv}
\bibliography{library.bib}
\end{document}